\documentclass[11pt,letterpaper]{article}

\usepackage{amssymb,amsthm,amsfonts,amscd,amsthm}
\usepackage[all,arc]{xy}
\usepackage{enumerate, bbm}
\usepackage{comment,caption}
\usepackage{mathrsfs}
\usepackage{tikz}
\usepackage[left=0.9in,top=0.9in,right=0.9in,bottom=0.9in]{geometry}
\usepackage{mathtools}
\usepackage{hyperref}
\usepackage{color}
\usepackage{graphicx}
\usepackage{enumitem} 
\graphicspath{ {TexImages/} }

\makeatletter
\renewenvironment{proof}[1][\proofname]{\par
  \vspace{-\topsep}
  \pushQED{\qed}%
  \normalfont
  \topsep8pt \partopsep0pt 
  \trivlist
  \item[\hskip\labelsep
        \itshape
    #1\@addpunct{.}]\ignorespaces
}{%
  \popQED\endtrivlist\@endpefalse
  \addvspace{6pt plus 6pt} 
}
\makeatother

\makeatletter
\def\thm@space@setup{%
  \thm@preskip=0.3cm
  \thm@postskip=0cm 
}
\makeatother

\newtheorem{thm}{Theorem}[section]

\newtheorem{prop}[thm]{Proposition}
\newtheorem{lem}[thm]{Lemma}

\newtheorem{problem}{Problem}

\newtheorem{prob}[thm]{Problem}

 \newtheorem{conjecture}{Conjecture}
 \newtheorem{theorem}{Theorem}

\theoremstyle{definition}

\hypersetup{
	colorlinks,
	citecolor=blue,
	filecolor=blue,
	linkcolor=blue,
	urlcolor=blue,
	linktocpage
}

\setenumerate[1]{label=\thesection.\arabic*.} 
\setenumerate[2]{label*=\arabic*.} 
\setlist[enumerate]{itemsep=2ex, topsep=2ex} 
\setlist[itemize]{itemsep=2ex, topsep=2ex}

\newcommand{\C}{\mathcal{C}}
\newcommand{\N}{\mbox{\rm N}}

\newcommand{\E}{\mathbb{E}}

\newcommand{\al}{\alpha}
\newcommand{\be}{\beta}

\newcommand{\lam}{\lambda}

\newcommand{\Om}{\Omega}

\newcommand{\Del}{\Delta}
\newcommand{\pa}{\partial}

\renewcommand{\l}{\left}
\renewcommand{\r}{\right}

\newcommand{\half}{\frac{1}{2}}

\newcommand{\sm}{\setminus}
\newcommand{\sub}{\subseteq}

\renewcommand{\c}[1]{\mathcal{#1}}

\newcommand{\rec}[1]{\frac{1}{#1}}
\newcommand{\f}[2]{\frac{#1}{#2}}
\newcommand{\floor}[1]{\l\lfloor #1\r\rfloor}
\newcommand{\ceil}[1]{\l\lceil #1\r\rceil}

\newcommand{\mr}[1]{\mathrm{#1}}

\newcommand{\ex}{\mr{ex}}

\title{Counting Hypergraphs with Large Girth}
\author{Sam Spiro\thanks{ Department of Mathematics, University of California, San Diego, 9500 Gilman Drive, La Jolla, CA 92093-0112, USA. E-mail: sspiro@ucsd.edu. This material is based upon work supported by the National Science Foundation Graduate Research Fellowship under Grant No. DGE-1650112.}\and
	Jacques Verstra\"ete\thanks{Department of Mathematics, University of California, San Diego, 9500 Gilman Drive, La Jolla, CA 92093-0112, USA. E-mail: jacques@ucsd.edu. Research supported by the National Science Foundation Awards DMS-1800332 and DMS-1952786,
		and by the Institute for Mathematical Research (FIM) of ETH Z\"urich.}}
\date{\today}

\begin{document}

	\maketitle

\vspace{-0.3in}

\begin{abstract}
Morris and Saxton~\cite{MS} used the method of containers to bound the number of $n$-vertex graphs with $m$ edges containing no $\ell$-cycles, and hence  graphs of girth more than $\ell$. We consider a generalization to $r$-uniform hypergraphs. The {\em girth} of a hypergraph $H$ is the minimum $\ell\ge 2$ such that there exist distinct vertices $v_1,\ldots,v_\ell$ and hyperedges $e_1,\ldots,e_\ell$ with $v_i,v_{i+1}\in e_i$ for all $1\le i\le \ell$. Letting ${\mbox{\rm N}}_m^r(n,\ell)$ denote the number of $n$-vertex $r$-uniform hypergraphs with $m$ edges and girth larger than $\ell$ and defining $\lambda = \lceil (r - 2)/(\ell - 2)\rceil$, we show  
\[ {\mbox{\rm N}}_m^r(n,\ell) \leq {\mbox{\rm N}}_m^2(n,\ell)^{r - 1 + \lambda}\]
which is tight when $\ell - 2 $ divides $r - 2$ up to a $1 + o(1)$ term in the exponent. This result is used to address the extremal problem for subgraphs of girth more than $\ell$ in random $r$-uniform hypergraphs.
\end{abstract}
\vspace{-1.1em}
\vspace{-2em}
\section{Introduction}
Let $\c{F}$ be a family of $r$-uniform hypergraphs, or $r$-graphs for short.  Define $\N^r(n,\c{F})$ to be the number of $\c{F}$-free $r$-graphs on $[n]:=\{1,\ldots,n\}$, and define $\N_m^r(n,\c{F})$ to be the number of $\c{F}$-free $r$-graphs on $[n]$ with exactly $m$ hyperedges.  If $\ex(n,\c{F})$ denotes the maximum number of hyperedges in an $\c{F}$-free $r$-graph on $[n]$, then it is not difficult to see that for $1 \leq m \leq \ex(n,\c{F})$,
\begin{align*}
\l(\f{\ex(n,\c{F})}{m}\r)^m \leq {\ex(n,\c{F}) \choose m} \le  & \;\;\N_m^r(n,\c{F})\le {{n \choose r} \choose m} \leq \l(\f{en^r}{m}\r)^m,
\end{align*}
and summing over $m$ one obtains $2^{\Omega(\ex(n,\c{F}))} = \N^r(n,\c{F}) = 2^{O(\ex(n,\c{F})\log n)}$. The state of the art for bounding $\N^r(n,\c{F})$ is the work of Ferber, McKinley, and Samotij~\cite{FMS} which shows that if $F$ is an $r$-uniform hypergraph with $\ex(n,F) = O(n^{\alpha})$ and $\al$ not too small, then
\begin{equation*}
\N^r(n,F)=2^{O(n^\al)},
\end{equation*}
and this result encompasses many of the earlier results in the area \cite{BNS, BS, CT, MS}.

There are relatively few families for which effective bounds for $\N_m^r(n,\c{F})$ are known.  One family where results are known is $\C_{[\ell]}=\{C_3,C_4,\dots,C_{\ell}\}$, the family of all graph cycles of length at most $\ell$.  Morris and Saxton implicitly proved the following in this setting:
\begin{thm}[\cite{MS}]\label{thm:MS}
	For $\ell\ge 3$ and $k = \lfloor \ell/2\rfloor$, there exists a constant $c=c(\ell)>0$ such that if $n$ is sufficiently large and $m\ge n^{1+1/(2k - 1)}(\log n)^2$, then
	\[\N_m^2(n,\C_{[\ell]}) \le e^{cm}(\log n)^{(k - 1)m}\l(\f{n^{1+1/k}}{m}\r)^{k m}.\]
\end{thm}
In the appendix we give a formal proof of this result. Theorem~\ref{thm:MS} generalizes earlier results of F\"{u}redi~\cite{F} when $\ell = 4$ and of Kohayakawa, Kreuter, and Steger~\cite{KKS}.
Erd\H{o}s and Simonovits~\cite{ES} conjectured for $\ell \geq 3$ and $k = \lfloor \ell/2 \rfloor$,
\begin{equation}\label{eq:es}
\ex(n,\C_{[\ell]}) = \Om(n^{1 + 1/k})
\end{equation}
 which is only known to hold for $\ell \in \{3,4,5,6,7,10,11\}$ -- see F\"{u}redi and Simonovits~\cite{FS} and also~\cite{V} for details. The truth of this conjecture would imply that the upper bound in Theorem \ref{thm:MS} is tight up to the exponent of $(\log n)^m$.

\medskip

In this paper we extend Theorem \ref{thm:MS} to $r$-graphs. For $\ell \geq 2$,
an $r$-graph $F$ is a {\em Berge $\ell$-cycle} if there exist distinct vertices $v_1,\ldots,v_\ell$ and distinct hyperedges $e_1,\ldots,e_\ell$ with $v_i,v_{i+1}\in e_i$ for all $1\le i\le \ell$.  In particular, a hypergraph $H$ is said to be {\em linear} if  it contains no Berge 2-cycle.
We denote by $\C_{\ell}^r$ the family of all $r$-uniform Berge $\ell$-cycles.
If $H$ is an $r$-graph containing a Berge cycle, then the {\em girth} of $H$ is the smallest $\ell\ge 2$ such that $H$ contains a Berge $\ell$-cycle.
Let $\C_{[\ell]}^r = \C_{2}^r \cup \C_3^r \cup \dots \cup \C_{\ell}^r$ denote the family of all $r$-uniform Berge cycles of length at most $\ell$.  With this $\C_{[\ell]}^2 = \C_{[\ell]}$, and an $r$-graph has girth larger than $\ell$ if and only if it is $\C_{[\ell]}^r$-free.  We again emphasize that hypergraphs with girth $\ell\ge 2$ are all linear.
We write $\N_m^r(n,\ell) := \N_m^r(n,\C_{[\ell]}^r)$ for the number of $n$-vertex $r$-graphs with $m$ edges and girth larger than $\ell$ and
$\N^r(n,\ell):=\N^r(n,\C_{[\ell]}^r)$ for the number of $n$-vertex $r$-graphs with girth larger than $\ell$.

Balogh and Li~\cite{BL} proved for all $\ell,r\ge3$ and $k = \lfloor \ell/2 \rfloor$,
\begin{equation*}
\N^r(n,\ell) = 2^{O(n^{1+ 1/k})}\label{eq:BLGirth}.
\end{equation*}
This upper bound would be tight up to a $n^{o(1)}$ term in the exponent if the following is true:

\begin{conjecture}\label{conj:girth}
	For all $\ell \geq 3$ and $r\ge 2$ and $k = \lfloor \ell/2\rfloor$, \[\ex(n,\C_{[\ell]}^r) = n^{1+ 1/k-o(1)}.\]
\end{conjecture}

Conjecture~\ref{conj:girth} holds for $\ell=3,4$ and $r \geq 3$ -- see~\cite{EFR,LV,RS,TV} -- but is open and evidently difficult for $\ell \geq 5$ and $r \geq 3$.
Gy\"{o}ri and Lemons~\cite{GL} proved $\ex(n,\C_{\ell}^r) = O(n^{1 + 1/k})$ with $k=\floor{\ell/2}$, so the conjecture concerns constructions of dense $r$-graphs of girth more than $\ell$.
The conjecture for $r = 2$ without the $o(1)$ is (\ref{eq:es}), and for each $r \geq 3$ is stronger than (\ref{eq:es}), as can be seen by forming a graph from an extremal $n$-vertex $r$-graph of girth more than $\ell$ whose edge set consists of an arbitrary pair of vertices from each hyperedge. We emphasize that the $o(1)$ term in Conjecture~\ref{conj:girth} is necessary for $\ell = 3$, due to the Ruzsa-Szemer\'{e}di Theorem~\cite{EFR,RS}, and for $\ell = 5$, due to work of Conlon, Fox, Sudakov and Zhao~\cite{CFSZ}.

\medskip

\begin{center}
{\sc 1.1 Counting $r$-graphs of large girth.}
\end{center}

In this work we simplify and refine the arguments of Balogh and Li~\cite{BL} to prove effective and almost tight bounds on $\N_m^r(n,\ell)$ relative to $\N_m^2(n,\ell)$.

\begin{thm}\label{thm:reduceGen}
	Let $\ell,r\ge 3$ and $\lambda = \lceil (r-2)/(\ell-2) \rceil$. Then for all $m,n\geq 1$,
\begin{equation}\label{eq:reduceGen}
	\N_m^r(n,\ell) \le \N_m^2(n,\ell)^{r-1+ \lambda}.
\end{equation}
\end{thm}

We note that (\ref{eq:reduceGen}) corrects a bound\footnote{Theorem 20 of \cite{PTTW} claims a stronger upper bound for $\N_m^r(n,4)$ than what we prove in Theorem~\ref{thm:reduceGen}, but we have confirmed with the authors that there was a subtle error in their proof.} which appears in \cite{PTTW}. The inequality (\ref{eq:reduceGen}) is essentially tight when $\ell - 2$ divides $r - 2$, due to standard probabilistic arguments (see for instance Janson, \L uczak and Rucinski~\cite{JLR}): it is possible to show that
when $m \le n^{1 + 1/(\ell - 1)}$, the uniform model of random $n$-vertex $r$-graphs with $m$ edges has girth larger than $\ell$ with probability at least $a^{-m}$ for some constant $a > 1$ depending only on $\ell$ and $r$. In particular, there exists some constants $b,c > 1$ such that for $m\le n^{1 + 1/(\ell - 1)}$ we have
\begin{equation}\label{eq:randexp}
\N_m^r(n,\ell) \geq a^{-m} {{n \choose r} \choose m} \geq b^{-m} (n^r/m)^m \ge b^{-m} (n^2/m)^{(r-1+ \frac{r-2}{\ell-2})m}\ge  c^{-m}\cdot  \N_m^2(n,\ell)^{r-1+ \frac{r-2}{\ell-2}},
\end{equation}
where the third inequality used $m\le n^{1+1/(\ell-1)}$ and the last inequality used the trivial bound $\N_m^2(n,\ell) \leq (en^2/m)^m$.  This shows that the bound of Theorem~\ref{thm:reduceGen} is best possible when $\ell-2$ divides $r-2$ up to a multiplicative error of $c^{-m}$ for some constant $c>1$.   We believe that (\ref{eq:randexp}) should define the optimal exponent, and propose the following conjecture:

\begin{conjecture}\label{conj:bestexponent}
For all $r \geq 2$, $\ell \geq 3$ and $m,n \geq 1$,
\begin{equation*}
	\N_m^r(n,\ell) \le \N_m^2(n,\ell)^{r-1+ \frac{r-2}{\ell-2}}.
\end{equation*}
\end{conjecture}

Theorem \ref{thm:reduceGen} shows that this conjecture is true when $\ell-2$ divides $r-2$, so the first open case of Conjecture \ref{conj:bestexponent} is when
$\ell = 4$ and $r = 3$. 

In the case that Berge $\ell$-cycles are forbidden instead of all Berge cycles of length at most $\ell$, we can prove an analog of Theorem~\ref{thm:reduceGen} with weaker quantitative bounds.  To this end, let $\N_{[m]}^r(n,\c{F})$ denote the number of $n$-vertex $\c{F}$-free $r$-graphs on at most $m$ hyperedges.

\begin{thm}\label{thm:reduceR3}
For each $\ell,r \geq 3$, there exists $c = c(\ell,r)$ such that
	\begin{equation*}
	\N_m^r(n,\C_{\ell}^r) \le 2^{c m}\cdot \N_{[m]}^2(n,C_{\ell})^{r!/2}.
	\end{equation*}
\end{thm}
We suspect that this result continues to hold with $\N_{[m]}^2(n,C_{\ell})$ replaced by $\N_{m}^2(n,C_{\ell})$.

\medskip

\begin{center}
{\sc 1.2 Subgraphs of random $r$-graphs of large girth.}
\end{center}

Denote by $H_{n,p}^r$ the $r$-graph obtained by including each hyperedge of $K_n^r$ independently and with probability $p$. Given a family of $r$-graphs $\c{F}$, let $\ex(H_{n,p}^r,\c{F})$ denote the size of a largest $\c{F}$-free subgraph of $H_{n,p}^r$.  Recall that a statement depending on $n$ holds \textit{asymptotically almost surely} or a.a.s.\ if it holds with probability tending to 1 as $n \rightarrow \infty$. A hypergraph of girth at least three is a linear hypergraph,
and it is not hard to show by a simple first moment calculation that if $p \geq n^{-r}\log n$, then a.a.s
\begin{equation*}
 \ex(H_{n,p}^r,\C_{[2]}^r) = \Theta(\min\{pn^r,n^2\}).
 \end{equation*}
 Our first result
essentially determines the a.a.s behavior of the number of edges in an extremal subgraph of $H_{n,p}^r$ of girth four.  In this theorem we omit the case $p < n^{-r + \f{3}{2}}$, as it is straightforward to show that a.a.s $\ex(H_{n,p}^r,\C_{[3]}^r) = \Theta(pn^r)$ when $p \geq n^{-r}\log n$ in this range.

\begin{thm}\label{thm:randTri}
Let $r \geq 3$. If $p\ge n^{-r+\f{3}{2}}(\log n)^{2r-3}$, then a.a.s.:
	\[p^{\f{1}{{2r-3}}}n^{2 - o(1)} \le \ex(H_{n,p}^r,\C_{[3]}^r)\le p^{\f{1}{{2r-3}}}n^{2 + o(1)}.\]
\end{thm}
Due to Theorems \ref{thm:reduceGen} and \ref{thm:randTri}, the number of linear triangle-free $r$-graphs with $n$ vertices and $m$ edges 
where $n^{3/2 + o(1)} \leq m \leq \ex(n,\C_{[3]}^r)=o(n^2)$ and $r\ge 3$ is:
\begin{equation*}
\N_m^r(n,3) = \N_m^2(n,3)^{2r-3 + o(1)} = \Bigl(\frac{n^2}{m}\Bigr)^{(2r - 3)m + o(m)} .
\end{equation*}
The authors and Nie~\cite{NSV} obtained bounds for $r$-uniform loose triangles\footnote{The loose triangle is the Berge triangle whose edges pairwise intersect in exactly one vertex.}, where for $r = 3$ the same essentially tight bounds as in Theorem~\ref{thm:randTri} were obtained, but for $r>3$ there remains a significant gap. In the case of subgraphs of girth larger than four,
Theorem~\ref{thm:reduceGen} allows us to generalize results of Morris and Saxton~\cite{MS} and earlier results of Kohayakawa, Kreuter and Steger~\cite{KKS} giving subgraphs of large girth in random graphs in the following way:

\begin{thm}\label{thm:randGirth}
	Let $\ell \geq 4$ and $r\ge 2$, and let $k=\floor{\ell/2}$ and $\lambda =\ceil{(r-2)/(\ell-2)}$.  Then a.a.s.:
	\[\ex(H_{n,p}^r,\C_{[\ell]}^r)\le \begin{cases}
	n^{1+\rec{2k-1}+o(1)} & n^{-r+1+\rec{\ell-1}}\le p <  n^{\f{-(r-1+\lambda)(k-1)}{2k-1}}(\log n)^{(r-1+\lambda)k},\\
	p^{\f{1}{{(r-1+\lambda)k}}}n^{1+\rec{k}+o(1)} & n^{\f{-(r-1+\lambda)(\ell-1-k)}{\ell-1}}(\log n)^{(r-1+\lambda)k}\le p \le 1.
	\end{cases}\]

		If Conjecture \ref{conj:girth} is true, then
	\[\ex(H_{n,p}^r,\C_{[\ell]}^r)\ge \begin{cases}
	n^{1+\rec{\ell-1} + o(1)} & n^{-r+1+\rec{\ell-1}}\le p < n^{\f{-(r-1)(\ell -1-k)}{\ell-1}},\\
	p^{\frac{1}{(r - 1)k}}n^{1 + \frac{1}{k} - o(1)} & n^{\f{-(r-1)(\ell -1-k)}{\ell-1}}\le p\le 1.
	\end{cases} \]
\end{thm}

  We note that the journal version of this article incorrectly states our first upper bound as $n^{1+\frac{1}{\ell-1}+o(1)}$ which is only true for $\ell$ even.  We  emphasize that there is a significant gap in the bounds of Theorem~\ref{thm:randGirth} due to the presence of $\lam$ in the exponent of $p$ in the upper bound and its absence in the lower bound, and this gap is closed by Theorem \ref{thm:randTri} when $\ell = 3$ by an improvement to the lower bound.
A similar phenomenon appears in recent work of Mubayi and Yepremyan~\cite{MY}, who determined the a.a.s value of the extremal function for loose even cycles in $H_{n,p}^r$ for all but a small range of $p$. It seems likely that the following conjecture is true:

\begin{conjecture}\label{conj:rand}
Let $\ell,r \geq 3$ and $k = \lfloor \ell/2 \rfloor$. Then there exists $\gamma = \gamma(\ell,r)$ such that a.a.s.:
\[ \ex(H_{n,p}^r,\C_{[\ell]}^r) = \left\{\begin{array}{ll}
n^{1 + \frac{1}{\ell - 1} + o(1)} & n^{-r + 1 + \frac{1}{\ell - 1}} \leq p < n^{-\frac{\gamma(\ell - 1 - k)}{\ell - 1}}, \\
p^{\frac{1}{\gamma k}} n^{1 + \frac{1}{k} + o(1)} &  n^{-\frac{\gamma(\ell - 1 - k)}{\ell - 1}} \leq p \leq 1.
\end{array}\right.\]
\end{conjecture}

Conjecture \ref{conj:bestexponent} suggests the possible value $\gamma(\ell,r) = r - 1 + (r - 2)/(\ell - 2)$, which is the correct value for $\ell = 3$ by Theorem 
\ref{thm:randTri}. We are not certain that this is the right value of $\gamma$ in general, even when $r = 3$ and $\ell = 4$, and more generally, 
Conjecture \ref{conj:girth} is an obstacle for $r \geq 3$ and $\ell \geq 5$. Theorem \ref{thm:randGirth}
shows that if $\gamma$ exists, then $(r - 1)k \leq \gamma \leq (r - 1 + \lambda)k$ provided Conjecture \ref{conj:girth} holds. 

\medskip

Letting $f(n,p) = \ex(H_{n,p}^3,\C_{[4]}^3)$, we plot the bounds of Theorem \ref{thm:randGirth} in Figure~\ref{fig:c4}, where the upper bound is in blue and the lower bound is in green. The truth of Conjecture \ref{conj:bestexponent} for $\ell = 4$ would imply the slightly better upper bound $f(n,p) \leq p^{1/5}n^{3/2 + o(1)}$.

\medskip

\begin{figure}[!ht]
	\begin{center}
		\includegraphics[width=4.5in]{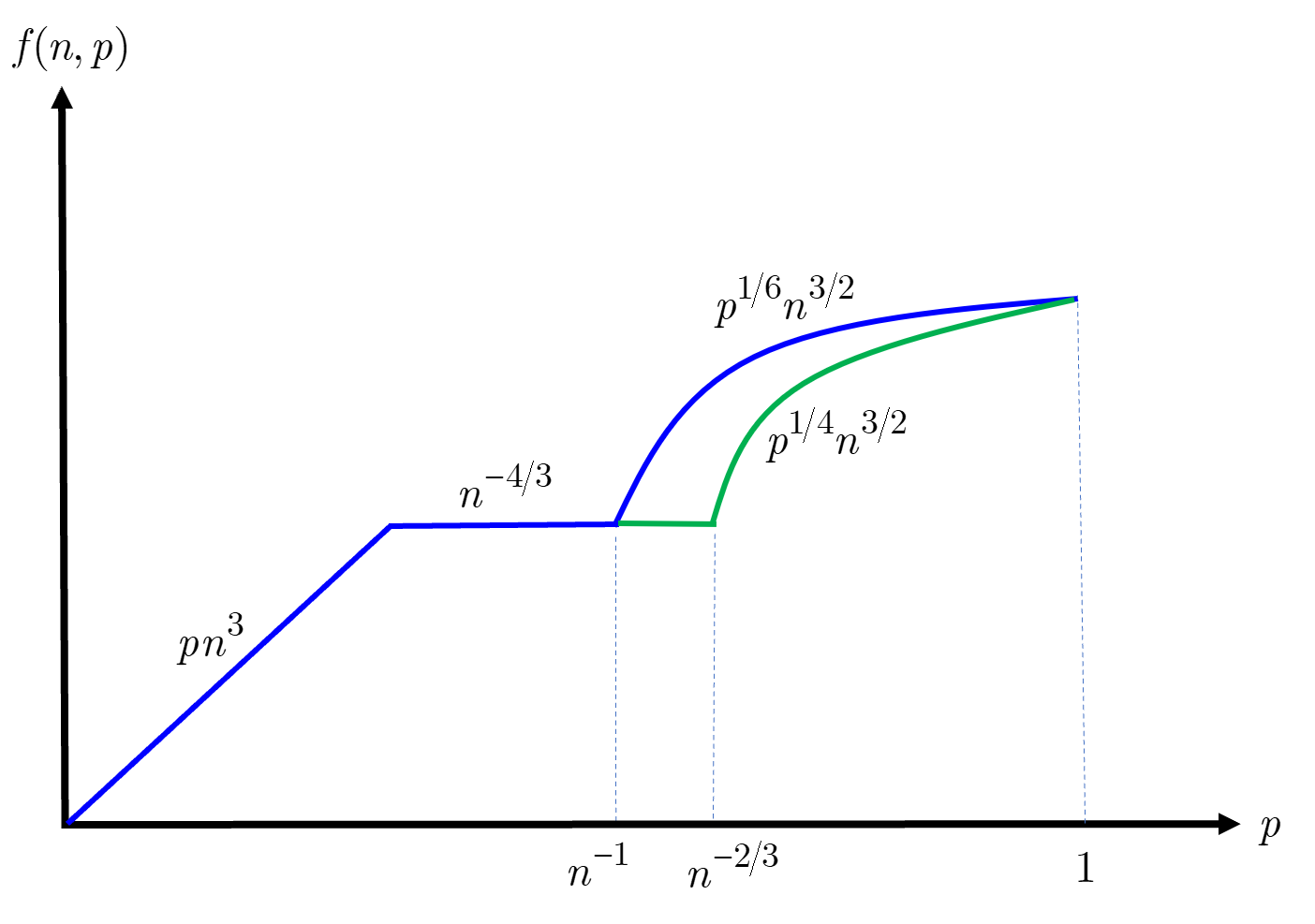}
		\caption{Subgraphs of $H_{n,p}^3$ of girth five}
		\label{fig:c4}
	\end{center}
\end{figure}

\medskip

\medskip

\medskip

\textbf{Notation}.  A set of size $k$ will be called a \textit{$k$-set}.  As much as possible, when working with a $k$-graph $G$ and an $r$-graph $H$ with $k<r$, we will refer to elements of $E(G)$ as edges and elements of $E(H)$ as hyperedges.  Given a hypergraph $H$ on $[n]$, we define the \textit{$k$-shadow} $\pa^k H$ to be the $k$-graph on $[n]$ consisting of all $k$-sets $e$ which lie in a hyperedge of $E(H)$.
If $G_1,\ldots,G_q$ are $k$-graphs on $[n]$, then $\bigcup G_i$ denotes the $k$-graph $G$ on $[n]$ which has edge set $\bigcup E(G_i)$.  

\section{Proof of Theorem~\ref{thm:reduceGen}}\label{sec:triangles}
As Balogh and Li~\cite{BL} observed, if $\ell \geq 3$ and $H$ has girth larger than $\ell$, then $H$ is uniquely determined by $\pa^2 H$, which we can view as the graph obtained by replacing each hyperedge of $H$ by a clique.  A key insight in proving Theorem~\ref{thm:reduceGen} is that we can replace each hyperedge of $H$ with a sparser graph $B$ and still uniquely recover $H$ from this graph.  To this end, we say that a graph $B$ is a \textit{book} if there exist cycles $F_1,\ldots,F_k$ and an edge $xy$ such that $B=\bigcup F_i$ and $E(F_i)\cap E(F_j)=\{xy\}$ for all $i\ne j$.  In this case we call the cycles $F_i$ the \textit{pages} of $B$ and we call the common edge $xy$ the \textit{spine} of $B$. The following lemma shows that if we replace each hyperedge in $H$ by a book on $r$ vertices which has small pages, then the vertex sets of books in the resulting graph are exactly the hyperedges of $H$.

\begin{lem}\label{lem:book}
	Let $H$ be an $r$-graph of girth larger than $\ell$.  If $\pa^2 H$ contains a book $B$ on $r$ vertices such that every page has length at most $\ell$, then there exists a hyperedge $e\in E(H)$ such that $V(B)= e$.
\end{lem}
\begin{proof}
	Let $F$ be a cycle in $\pa^2 H$ with $V(F)=\{v_1,\ldots,v_p\}$ such that $v_iv_{i+1}\in E(\pa^2 H)$ for $i<p$ and $v_1v_p\in E(\pa^2 H)$.  If $p\le \ell$ we claim that there exists an $e\in E(H)$ such that $V(F)\sub e$.  Indeed, by definition of $\pa^2 H$ there exists some hyperedge $e_i\in E(H)$ with $v_i,v_{i+1}\in e_i$ for all $i<p$ and some hyperedge $e_p$ with $v_1,v_p\in e_p$.  If all of these $e_i$ hyperedges are equal then we are done, so we may assume $e_1\ne e_{p}$.  Define $i_1$ to be the largest index such that $e_i=e_1$ for all $i\le i_1$, define $i_2$ to be the largest index so that $e_i=e_{i_1+1}$ for all $i_1<i\le i_2$, and so on up to $i_q=p$, and note that $2\le q\le p$ since $e_1\ne e_{p}$.  If all the $e_{i_j}$ hyperedges are distinct, then they form a Berge $q$-cycle in $H$ since $v_{1+i_j}\in e_{i_j}\cap e_{1+i_{j}}= e_{i_j}\cap e_{i_{j+1}}$ for all $j$, a contradiction.  Thus we can assume $e_{i_j}=e_{i_{j'}}$ for some $j<j'$.  We can further assume that $e_{i_s}\ne e_{i_{s'}}$ for any $j\le s<s'<j'$, as otherwise we could replace $j,j'$ with $s,s'$.  Finally note that $j<j'-1$, as otherwise we would have $e_{i_j}=e_{i_{j'}}=e_{i_j+1}$, contradicting the maximality of $i_j$.  We conclude that the distinct hyperedges $e_{i_j},e_{i_{j+1}},\ldots,e_{i_{j'-1}}$ form a Berge $(j' - j)$-cycle with $2\le j'-j\le \ell$ in $H$, a contradiction. This proves the claim.
	
	Now let $B$ be a book with spine $xy$ and pages $F_1,\ldots,F_k$ of length at most $\ell$.  By the claim there exist  hyperedges $e_1,\ldots,e_k\in E(H)$ such that $V(F_i)\sub e_i$ for all $i$, and in particular $x,y\in e_i$ for all $i$. Because $H$ is linear, this implies that all of these hyperedges are equal and we have $V(B)\sub e_1$.  If $B$ has $r$ vertices, then we further have $V(B)=e_1$.
\end{proof}

We now complete the proof of Theorem \ref{thm:reduceGen}. With $\lam:=\ceil{(r-2)/(\ell-2)}$ we observe for all $\ell,r\ge 3$ that there exists a book graph $B$ on $r$ vertices $\{x_1,\ldots,x_r\}$ with $r-1+\lam$ edges $f_1,\ldots,f_{r-1+\lam}$.  Indeed if $\ell-2$ divides $r-2$ one can take $\lam$ copies of $C_\ell$ which share a common edge, and otherwise one can take $\lam-1$ copies of $C_\ell$ and a copy of $C_p$ with $p=r-(\lambda-1)(\ell-2)\ge 3$.  From now on we let $B$ denote this book graph.  If $f_i=\{x_j,x_{j'}\}\in E(B)$ and $e=\{v_1,\ldots,v_r\}\sub [n]$ is any $r$-set with $v_1<\cdots<v_r$, define $\phi_{i}(e)=\{v_j,v_{j'}\}$.  If $H$ is an $r$-graph on $[n]$, define $\phi_{i}(H)$ to be the graph on $[n]$ which has all edges of the form $\phi_{i}(e)$ for $e\in E(H)$; so in particular $\bigcup \phi_i(H)$ is the graph obtained by replacing each hyperedge of $H$ with a copy of $B$.

Let $\c{H}_{m,n}$ denote the set of $r$-graphs on $[n]$ with $m$ hyperedges and girth more than $\ell$, and let $\c{G}_{m,n}$ be the set of graphs on $[n]$ with $m$ edges and girth more than $\ell$. We claim that $\phi_{i}$ maps $\c{H}_{m,n}$ to $\c{G}_{m,n}$.  Indeed, if $H\in \c{H}_{m,n}$ then each hyperedge of $H$ contributes a distinct edge to $\phi_{i}(H)$ since $H$ is linear, so $e(\phi_i(H))=e(H)=m$.  One can show that if $\phi_i(e_1),\ldots,\phi_i(e_p)$ form a $p$-cycle in $\phi_i(H)$, then $e_1,\ldots,e_p$ form a Berge $p$-cycle in $H$; so $H\in \c{H}_{m,n}$ implies $\phi_i(H)$ does not contain a cycle of length at most $\ell$.
	
Let $\c{G}_{m,n}^{t} = \{(G_1,G_2,\dots,G_{t}) : G_i \in \c{G}_{m,n}\}$. Then we define a map 
$\phi:\c{H}_{m,n}\to \c{G}_{m,n}^{r-1+\lam}$ by 
\[ \phi(H)=(\phi_1(H),\ldots,\phi_{r-1+\lam}(H)).\]
We claim that this map is injective.  Indeed, fix some $H\in \c{H}_{m,n}$ and let $\c{B}(G)$ denote the set of books $B$ in the graph $G:=\bigcup \phi_i(H)\sub \pa^2 H$.  By definition of $\phi$ we have $E(H)\sub \c{B}(G)$ for all $H$.  Moreover, if $H\in \c{H}_{m,n}$ then Lemma~\ref{lem:book} implies $\c{B}(G)\sub E(H)$.  Thus $E(H)$ (and hence $H$) is uniquely determined by $G$, which is itself determined by $\phi(H)$, so the map is injective.  In total we conclude
	\[\N_m^r(n,\ell)=|\c{H}_{m,n}|\le|\c{G}_{m,n}^{r-1+\lambda}|= \N_m^2(n,\ell)^{r-1+\lambda},\]
proving Theorem~\ref{thm:reduceGen}. \hfill $\blacksquare$

\section{Proof of Theorem~\ref{thm:reduceR3}}\label{sec:noTriangles}
For arbitrary hypergraphs $H$, the map $\phi(H)=\pa^{r-1}H$ (let alone the map to $\pa^2 H$) is not injective.  However, we will show that this map is ``almost'' injective when considering $H$ which are $\C_\ell^r$-free. To this end, we say that a set of vertices $\{v_1,\ldots,v_r\}$ is a \textit{core set} of an $r$-graph $H$ if there exist distinct hyperedges $e_1,\ldots,e_r$ with $\{v_1,\ldots,v_r\}\sm\{v_i\}\sub e_i$ for all $i$.  The following observation shows that core sets are the only obstruction to $\phi(H)=\pa^{r-1}H$ being injective.

\begin{lem}\label{lem:shadow}
	Let $H$ be an $r$-graph.  If $\{v_1,\ldots,v_r\}$ induces a $K_r^{r-1}$ in $\pa^{r-1} H$, then either $\{v_1,\ldots,v_r\}\in E(H)$ or $\{v_1,\ldots,v_r\}$ is a core set of $H$.
\end{lem}
\begin{proof}
	By assumption of $\{v_1,\ldots,v_r\}$ inducing a $K_r^{r-1}$ in $\pa^{r-1} H$, for all $i$ there exist $e'_i\in E(\pa^{r-1}H)$ with $e'_i=\{v_1,\ldots,v_r\}\sm \{v_i\}$.  By definition of $\pa^{r-1}H$, this means there exist (not necessarily distinct) $e_i\in E(H)$ with $e_i\supseteq e_i'=\{v_1,\ldots,v_r\}\sm \{v_i\}$.  Given this, either $e_i=\{v_1,\ldots,v_r\}$ for some $i$, or all of the $e_i$ distinct, in which case $\{v_1,\ldots,v_r\}$ is a core set of $H$.  In either case we conclude the result.

\end{proof}

We next show that $\C_\ell^r$-free $r$-graphs have few core sets.
\begin{lem}\label{lem:BKd}
	Let $\ell,r\ge 3$ and let $H$ be a $\C_{\ell}^r$-free $r$-graph with $m$ hyperedges.  The number of core sets in $H$ is at most $\ell^{2}r^2 m$.
\end{lem}
\begin{proof}
	We claim that $H$ contains no core sets if $\ell\le r$.  Indeed, assume for contradiction that $H$ contained a core set $\{v_1,\ldots,v_r\}$ with distinct hyperedges $e_i\supseteq \{v_1,\ldots,v_r\}\sm \{v_i\}$.  It is not difficult to see that the hyperedges $e_1,\ldots,e_\ell$ form a Berge $\ell$-cycle, a contradiction to $H$ being $\C_{\ell}^r$-free.  Thus from now on we may assume $\ell>r$.
	
	Let $\c{A}_1$ denote the set of core sets in $H$, and for any $\c{A}'\sub \c{A}_1$ and $(r-1)$-set $S$, define $d_{\c{A}'}(S)$ to be the number of core sets $A\in \c{A}'$ with $S\sub A$.  Observe that $d_{\c{A}_1}(S)>0$ for at most ${r\choose r-1}m=rm$ $(r-1)$-sets $S$, since in particular $S$ must be contained in a hyperedge of $H$.
	
	Given $\c{A}_i$, define $\c{A}'_i\sub \c{A}_i$ to be the core sets $A\in \c{A}_i$ which contain an $(r-1)$-set $S$ with $d_{\c{A}_i}(S)\le \ell r$, and let $\c{A}_{i+1}=\c{A}_i\sm \c{A}'_i$.  Observe that $|\c{A}'_i|\le \ell r\cdot r m$ since each $(r-1)$-set $S$ with $d_{\c{A}_i}(S)>0$ is contained in at most $\ell r$ elements of $\c{A}'_i$.  In particular,
	\begin{equation}
		\label{eq:A}|\c{A}_1|\le (\ell-r)\cdot \ell r^2m+|\c{A}_{\ell-r+1}|\le \ell^{2}r^2 m+|\c{A}_{\ell-r+1}|.
	\end{equation}
	
	Assume for the sake of contradiction that $\c{A}_{\ell-r+1}\ne \emptyset$.   We prove by induction on $r\le i\le \ell$ that one can find distinct vertices $v_1,\ldots,v_i$ and distinct hyperedges $e_1,\ldots,e_{i-1},\tilde{e}_i$ such that $v_j,v_{j+1}\in e_{j}$ for $1\le j<i$ and $v_1,v_i\in \tilde{e}_i$, and such that $\{v_i,v_{i-1},\ldots,v_{i-r+2},v_1\}\in \c{A}_{\ell-i+1}$.  For the base case, consider any $\{v_r,v_{r-1},\ldots,v_1\}\in \c{A}_{\ell-r+1}$.  As this is a core set, there exist distinct hyperedges $e_j\supseteq \{v_1,\ldots,v_r\}\sm \{v_{j+2}\}$ and $\tilde{e}_r\supseteq\{v_1,\ldots,v_r\}\sm \{v_{2}\}$, proving the base case of the induction. 
	
	Assume that we have proven the result for $i<\ell$.  By assumption of $\{v_i,v_{i-1},\ldots,v_{i-r+2},v_1\}\in \c{A}_{\ell-i+1}$, we have $\{v_i,v_{i-1},\ldots,v_{i-r+2},v_1\}\notin \c{A}_{\ell-i}'$, so there exists a set of vertices $\{u_1,\ldots,u_{\ell r+1}\}$ such that $\{v_i,v_{i-1},\ldots,v_{i-r+3},v_1,u_j\}\in \c{A}_{\ell-i}$ for all $j$.  Because $|\bigcup_{k=1}^{i-1} e_k|\le \ell r$, there exists some $j$ such that $u_j\notin \bigcup_{k=1}^{i-1} e_k$.  For this $j$, let $v_{i+1}:=u_j$ and let $e_i,\tilde{e}_{i+1}$ be distinct hyperedges containing $v_i,v_{i+1}$ and $v_1,v_{i+1}$ respectively, which exist by assumption of this being a core set.  Note that $v_{i+1}$ is distinct from every other $v_{i'}$ since $v_{i'}\in \bigcup_{k=1}^{i-1} e_k$ for $i'\le i$, and similarly the hyperedges $e_i,\tilde{e}_{i+1}$ are distinct from every hyperedge $e_{i'}$ with $i'<i$ since these new hyperedges contain $v_{i+1}\notin \bigcup_{k=1}^{i-1} e_k$.  This proves the inductive step and hence the claim.  The $i=\ell$ case of this claim implies that $H$ contains a Berge $\ell$-cycle, a contradiction.  Thus $\c{A}_{\ell-r+1}=\emptyset$, and the result follows by \eqref{eq:A}.
\end{proof}
Combining these two lemmas gives the following result, which allows us to reduce from $r$-graphs to $(r-1)$-graphs.  We recall that $\N_{[m]}^r(n,\c{F})$ denotes the number of $n$-vertex $\c{F}$-free $r$-graphs on at most $m$ hyperedges.
\begin{prop}
		For each $\ell,r \geq 3$, there exists $c = c(\ell,r)$ such that
		\begin{equation*}
			\N_{[m]}^r(n,\C_{\ell}^r) \le 2^{c m}\cdot \N_{[m]}^r(n,\C_{\ell}^{r-1})^{r}.
		\end{equation*}
\end{prop}
\begin{proof}
	If $e=\{v_1,v_2,\ldots,v_r\}\sub [n]$ is any $r$-set with $v_1<v_2<\cdots<v_r$, let $\phi_i(e)=\{v_1,\ldots,v_r\}\sm \{v_i\}$.  Given an $r$-graph $H$ on $[n]$, let $\phi_i(H)$ be the $(r-1)$-graph on $[n]$ with edge set $\{\phi_i(e):e\in E(H)\}$, and define $\phi(H)=(\phi_1(H),\phi_2(H),\ldots,\phi_r(H))$ and $\psi(H)=(\phi(H),E(H))$.  Observe that $\bigcup \phi_i(H)=\pa^{r-1}H$.   Let $\c{H}_{[m],n}$ denote the set of all $r$-graphs on $[n]$ with at most $m$ hyperedges which are $\C_{\ell}^r$-free, and let $\phi(\c{H}_{[m],n}),\psi(\c{H}_{[m],n})$ denote the image sets of $\c{H}_{[m],n}$ under these respective maps.  Observe that $\psi$ is injective since it records $E(H)$, so it suffices to bound how large $\psi(\c{H}_{[m],n})$ can be.
	
	Let $\c{G}_{[m],n}$ denote the set of $(r-1)$-graphs on $[n]$ which have at most $m$ edges and which are $\C_\ell^{r-1}$-free.  It is not difficult to see that $\phi(\c{H}_{[m],n})\sub \c{G}_{[m],n}^{r}$.  We observe by Lemmas~\ref{lem:shadow} and \ref{lem:BKd} that for any $(G_1,G_2,\ldots,G_r)\in \phi(\c{H}_{[m],n})$, say with $\phi(H)=(G_1,\ldots,G_r)$, there are at most $(1+\ell^2r^2)m$ copies of $K_r^{r-1}$ in $\bigcup G_i=\pa^{r-1}H$.  We also observe that if $((G_1,G_2,\ldots,G_r),E)\in \psi(\c{H}_{[m],n})$, then $E$ is a set of at most $m$ copies of $K_r^{r-1}$ in $\bigcup G_i$. Thus given any $(G_1,\ldots,G_r)\in \phi(\c{H}_{[m],n})\sub \c{G}_{[m],n}^{r}$, there are at most $2^{(1+\ell^2r^2)m}$ choices of $E$ such that $((G_1,\ldots,G_r),E)\in \psi(\c{H}_{[m],n})$.  We conclude that
	\[\N_{[m]}(n,\C_\ell^r)=|\c{H}_{[m],n}|\le |\c{G}_{[m],n}|^{r}\cdot 2^{(1+\ell^2r^2)m}=\N_{[m]}^r(n,\C_{\ell}^{r-1})^{r}\cdot 2^{(1+\ell^2r^2)m},\]
	proving the result.
\end{proof}
Applying this proposition repeatedly gives $\N_{[m]}^r(n,\C_{\ell}^r)\le 2^{cm} \N_{[m]}^2(n,C_\ell)^{r!/2}$.  Combining this with the trivial inequality $\N_{m}^r(n,\C_{\ell}^r)\le \N_{[m]}^r(n,\C_{\ell}^r)$ gives Theorem~\ref{thm:reduceR3}. \hfill $\blacksquare$

\section{Proof of Theorems~\ref{thm:randTri} and~\ref{thm:randGirth}}\label{sec:random}

To prove that our bounds hold a.a.s.,\ we use the Chernoff bound~\cite{ProbMeth}.
\begin{prop}[\cite{ProbMeth}]\label{prop:Chern}
	Let $X$ denote a binomial random variable with $N$ trials and probability $p$ of success.  For any $\epsilon>0$ we have $\Pr[|X-pN|>\epsilon pN]\le 2\exp(-\epsilon^2 pN/2)$.
\end{prop}

{\bf Proof of the upper bounds in Theorem~\ref{thm:randGirth}.} Let
 \[p_0 = n^{-\f{(r - 1 + \lambda)(k-1)}{2k - 1}} (\log n)^{(r - 1 + \lambda)k}. \]
 For $p\ge p_0$, define
 \[m=p^{\rec{(r-1+\lam)k}}n^{1+\rec{k}}\log n,\]
and note that this is large enough to apply Theorem~\ref{thm:MS} for $p \geq p_0$.  Let $Y_m$ denote the number of subgraphs of $H_{n,p}^r$ which are $\C_{[\ell]}^r$-free and have exactly $m$ edges, and note that $\ex(H_{n,p}^r,\C_{[\ell]}^r)\ge m$ if and only if $Y_m\ge 1$.  By Markov's inequality, Theorem \ref{thm:reduceGen}, and Theorem~\ref{thm:MS}:
\begin{eqnarray*}
\Pr[Y_m\ge 1] &\le&\E[Y_m]=p^m\cdot \N_m^r(n,\ell) \\ &\le& p^m \cdot \N_m^2(n,\ell)^{r-1+\lambda}\\
&\leq& \Bigl(p^{\frac{1}{r - 1 + \lambda}} e^{c} (\log n)^{k - 1}\Bigl(\frac{n^{1 + \f{1}{k}}}{m}\Bigr)^{k}\Bigr)^{m(r - 1 + \lambda)} \\
&=& \Bigl(\frac{e^c}{\log n}\Bigr)^{m(r - 1 + \lambda)}.
\end{eqnarray*}
The right hand side converges to zero, so for $p \ge p_0$, a.a.s:
\[ \ex(H_{n,p}^r,\C_{[\ell]}^r) < m.\]

As $\E[\ex(H_{n,p}^r,\C_{[\ell]}^r)]$ is non-decreasing in $p$, the bound \[\ex(H_{n,p}^r,\C_{[\ell]}^r)<n^{1+\rec{2k-1}}(\log n)^2\]
continues to hold a.a.s.\  for all $p<p_0$.   \hfill $\blacksquare$

\medskip

{\bf Proof of the upper bound in Theorem~\ref{thm:randTri}.}  This proof is almost identical to the previous, so we omit some of the redundant details.  Let $m=p^{\rec{2r-3}}n^2\log n$ and let $Y_m$ denote the number of subgraphs of $H_{n,p}^r$ which are $\C_{[\ell]}^r$-free and have exactly $m$ edges.   By Markov's inequality, Theorem \ref{thm:reduceGen}, and the trivial bound $\N_m^2(n,3)\le {n^2\choose m}$ which is valid for all $m$, we find for all $p$
\[\Pr[Y_m\ge 1]\le p^m (en^2/m)^{(2r-3)m}=(e/\log n)^m.\]
This quantity converges to zero, so we conclude the result by the same reasoning as in the previous proof. \hfill $\blacksquare$

This proof shows that for all $p$ we have $\E[\ex(H_{n,p}^r,\C_{[\ell]}^r)]<p^{\rec{2r-3}}n^2\log n$.  However, for $p\le  n^{-r+3/2}$ this is weaker than the trivial upper bound $\E[\ex(H_{n,p}^r,\C_{[\ell]}^r)]\le p{n\choose r}$.


\medskip

{\bf Proof of the lower bounds in Theorem \ref{thm:randGirth}.} We use homomorphisms similar to
Foucaud, Krivelevich and Perarnau~\cite{FKP} and Perarnau and Reed~\cite{PR}. If $F$ and $F'$ are hypergraphs and $\chi : V(F) \rightarrow V(F')$ is any map, we let $\chi(e) = \{\chi(u) : u \in e\}$ for any $e\in E(F)$. For two $r$-graphs $F$ and $F'$, a map $\chi:V(F)\to V(F')$ is a \textit{homomorphism} if $\chi(e)\in E(F')$ for all $e\in E(F)$, and $\chi$ is a {\em local isomorphism}
if $\chi$ is a homomorphism and $\chi(e)\ne \chi (f)$ whenever $e,f\in E(F)$ with $e\cap f \ne \emptyset$. A key lemma is the following:

\begin{lem}\label{lem:homB}
If $F \in  \C_{[\ell]}^r$ and $\chi : F \rightarrow F'$ is a local isomorphism, then $F'$ has girth at most $\ell$.
\end{lem}

\begin{proof} Let $F$ be a Berge $p$-cycle with $p \leq \ell$ and $E(F) = \{e_1,e_2,\dots,e_p\}$. Then there exist distinct vertices $v_1,v_2,\dots,v_p$ such that $v_i \in e_i \cap e_{i + 1}$ for $i < p$ and $v_p \in e_p \cap e_1$. First assume there exists $i \ne j$ such that $\chi(e_i)=\chi(e_j)$.  By reindexing, we can assume $\chi(e_1)=\chi(e_{k})$ for some $k>1$, and further that $\chi(e_i)\ne \chi(e_j)$ for any $1\le i<j<k$.  Note that $k\ge 3$ since $e_1\cap e_2\ne \emptyset$ and $\chi$ is a local isomorphism.  If we also have $\chi(v_i)\ne \chi(v_j)$ for all $1\le i<j<k$, then $\chi(v_i) \in \chi(e_i) \cap \chi(e_{i + 1})$ for $i < k - 1$ and $\chi(v_{k - 1}) \in \chi(e_{k - 1}) \cap \chi(e_1)$, so
$\chi(e_1),\chi(e_2),\dots,\chi(e_{k - 1})$ is the edge set of a Berge $(k - 1)$-cycle in $F'$ as required.

Suppose $\chi(v_i)=\chi(v_j)$ for some $1\le i<j<k$, and as before we can assume there exists no $i\le i'<j'<j$ with $\chi(v_{i'})=\chi(v_{j'})$.  Then $\chi(v_i),\chi(v_{i+1}),\ldots,\chi(v_{j-1})$ are distinct vertices with $\chi(v_h) \in \chi(e_h) \cap \chi(e_{h + 1})$ for $i \leq h < j - 1$ and
$\chi(v_{j - 1}) \in \chi(e_{j - 1}) \cap \chi(e_1)$. Note that $\chi(v_i)\ne \chi(v_{i+1})$ since this would imply $|\chi(e_{i})|<r$, contradicting that $\chi$ is a homomorphism, so $j>i+1$.  Thus the hyperedges $\chi(e_i),\chi(e_{i+1}),\ldots,\chi(e_{j-1})$ form a Berge $(j - i)$-cycle in $F'$ with $j-i\ge 2$ as desired.

This proves the result if $\chi(e_i)=\chi(e_j)$ for some $i\ne j$.  If this does not happen and the $\chi(v_i)$ are all distinct, then $F'$ is a Berge $p$-cycle, and if $\chi(v_i) = \chi(v_j)$ then the same proof as above gives a Berge $(j - i)$-cycle in $F'$.
\end{proof}

The following lemma allows us to find a relatively dense subgraph of large girth in any $r$-graph whose maximum $i$-degree is not too large, where the \textit{$i$-degree} of an $i$-set $S$ is the number of hyperedges containing $S$.

\begin{lem}\label{lem:homGen}
	Let $\ell,r\ge 3$ and let $H$ be an $r$-graph with maximum $i$-degree $\Del_i$ for each $i \geq 1$.  If $t\ge r^24^r \Del_i^{1/(r-i)}$ for all $i \geq 1$, then $H$ has a subgraph $H'$ of girth larger than $\ell$  with
\[ e(H') \ge \ex(t,\C_{[\ell]}^r)t^{-r}\cdot e(H).\]
\end{lem}
\begin{proof}
	Let $J$ be an extremal $\C_{[\ell]}^r$-free $r$-graph on $t$ vertices and $\chi:V(H)\to V(J)$ chosen uniformly at random.  Let $H'\sub H$ be the random subgraph which keeps the hyperedge $e\in E(H)$ if
	\begin{itemize}
		\item[$(1)$] $\chi(e)\in E(J)$, and
		\item[$(2)$] $\chi(e)\ne \chi(f)$ for any other $f\in E(H)$ with $|e\cap f|\ne 0$.
	\end{itemize}
	
	We claim that $H'$ is $\C_{[\ell]}^r$-free.  Indeed, assume $H'$ contained a subgraph $F$ isomorphic to some element of $\C_{[\ell]}^r$.  Let $F'$ be the subgraph of $J$ with $V(F')=\{\chi(u):u\in V(F)\}$ and $E(F')=\{\chi(e):e\in E(F)\}$, and note that $F\sub H'$ implies that each hyperedge of $F$ satisfies (1), so every element of $E(F')$ is a hyperedge in $J$.  By conditions (1) and (2), $\chi$ is a local isomorphism from $F$ to $F'$.  By Lemma~\ref{lem:homB}, $F'\sub J$ contains a Berge cycle of length at most $\ell$, a contradiction to $J$ being $\C_{[\ell]}^r$-free.
	
	It remains to compute $\E[e(H')]$.  Given $e\in E(H)$, let $A_1$ denote the event that (1) is satisfied, let $E_i=\{f\in E(H):|e\cap f|=i\}$, and let $A_2$ denote the event that $\chi(f)\not\sub \chi(e)$ for any $f\in \bigcup_{i} E_i$, which in particular implies (2) for the hyperedge $e$.  It is not too difficult to see that $\Pr[A_1]= r!e(J)t^{-r}$, and that for any $f\in E_i$ we have $\Pr[\chi(f)\sub \chi(e)|A_1]=(r/t)^{r-i}$.  Note for each $i \geq 1$ that $|E_i|\le 2^r \Del_i$, as $e$ has at most $2^r$ subsets of size $i$ each of $i$-degree at most $\Del_i$.  Taking a union bound we find \[\Pr[A_2|A_1]\ge 1-\sum_{i = 1}^r |E_i|(r/t)^{r-i}\ge 1- \sum_{i = 1}^r 2^r \Del_i(r/t)^{r-i}\ge 1-\sum_{i = 1}^r r^{-1}2^{-r}\ge \half,\]
	where the second to last inequality used $(r 4^r)^{i-r}\ge r^{-1}4^{-r}$ for $i \leq r$. Consequently
	\[\Pr[e\in E(H')]=\Pr[A_1]\cdot \Pr[A_2|A_1]\ge r! e(J)t^{-r}\cdot \half \ge e(J)t^{-r},\]
	and linearity of expectation gives $\E[e(H')]\ge e(J)t^{-r}\cdot e(H)=\ex(t,\C_{[\ell]}^r)t^{-r}\cdot e(H)$.  Thus there exists some $\C_{[\ell]}^r$-free subgraph $H'\sub H$
with at least $\ex(t,\C_{[\ell]}^r)t^{-r}\cdot e(H)$ hyperedges.
\end{proof}

By the Chernoff bound one can show for
\[p\ge p_1:=n^{\f{-(r-1)(\ell-1-k)}{\ell-1}}\]
that a.a.s. $H_{n,p}^r$ has maximum $i$-degree at most $\Theta(pn^{r-i})$ for all $i$.   If Conjecture \ref{conj:girth} is true,
then a.a.s for $p \geq p_1$ Lemma \ref{lem:homGen} with $t = \Theta(p^{1/(r-1)}n)$ gives:
\[\ex(H_{n,p}^r,\C_{[\ell]}^r) = \Omega(t^{-r} \ex(t,\C_{[\ell]}^r) pn^{r}) = p^{\rec{(r - 1)k}}n^{1 + \rec{k} - o(1)}.\]
This proves the lower bound in Theorem \ref{thm:randGirth}.  \hfill $\blacksquare$

\medskip

{\bf Proof of the lower bound in Theorem \ref{thm:randTri}.}  We use the following variant of Lemma~\ref{lem:homGen}:
\begin{lem}\label{lem:homTri}
	Let $H$ be an $r$-graph and let $R_{\ell,v}(H)$ be the number of Berge $\ell$-cycles in $H$ on $v$ vertices.  For all $t\ge 1$, $H$ has a subgraph $H'$ of girth larger than $3$ with
	\[ e(H') \ge \l(e(H)t^{2-r}-\sum_{\ell=2}^3\sum_{v}t^{2-v}R_{\ell,v}(H)\r)e^{-c\sqrt{\log t}},\]
	where $c>0$ is an absolute constant.
\end{lem}
\begin{proof}
	By work of Ruzsa and Szemeredi~\cite{RS} and Erd\H{o}s, Frankl, R\"odl~\cite{EFR}, it is known for all $t$ that there exists a $\C_{[3]}^r$-free $r$-graph $J$ on $t$ vertices with $t^2e^{-c\sqrt{\log t}}$ hyperedges. Choose a map $\chi:V(H)\to V(J)$ uniformly at random and define $H''\sub H$ to be the subgraph which keeps a hyperedge $e=\{v_1,\ldots,v_r\}\in E(H)$ if and only if $\chi(e)\in E(J)$.
	
	We claim that if $e_1,e_2,e_3$ form a Berge triangle in $H''$, then $\chi(e_1)=\chi(e_2)=\chi(e_3)$.  Observe that if $v_1,v_2,v_3$ are vertices with $v_i\in e_i\cap e_{i+1}$, then we must have e.g. $\chi(v_1)\ne \chi(v_2)$, as otherwise $|\chi(e_2)|<r$. Because $J$ is linear we must have $|\chi(e_i)\cap \chi(e_j)|\in\{1,r\}$. These hyperedges can not all intersect in 1 vertex since this together with the distinct vertices $\chi(v_1),\chi(v_2),\chi(v_3)$ defines a Berge triangle in $H''$, so we must have say $\chi(e_1)=\chi(e_2)$. But this means $\chi(v_3),\chi(v_2)$ are distinct vertices in $\chi(e_1)=\chi(e_2)$ and $\chi(e_3)$, so $|\chi(e_1)\cap \chi(e_3)|>1$ and we must have $\chi(e_1)=\chi(e_3)$ as desired.
	
	The probability that a given Berge triangle $C$ on $v$ vertices in $H$ maps to a given hyperedge in $J$ is at most $(r/t)^{v}$ (since this is the probability that every vertex of $C$ maps into the edge of $J$).  By linearity of expectation, $H''$ contains at most $\sum_v R_{3,v}(H)e(J)(r/t)^v$ Berge triangles in expectation.  An identical proof shows that $H''$ contains at most $\sum_v R_{2,v}(H)e(J)(r/t)^v$ Berge 2-cycles in expectation.  We can then delete a hyperedge from each of these Berge cycles in $H''$ to find a subgraph $H'$ with
	\[\E[e(H')]\ge e(J)t^{-r}\cdot e(H)-\sum_{\ell=2}^3\sum_vR_{\ell,v}(H)e(J)(r/t)^v.\]
	The result follows since $e(J)=t^2e^{-c\sqrt{\log t}}$.
\end{proof}

We now prove the lower bound in Theorem \ref{thm:randTri}. By Markov's inequality one can show that a.a.s. $R_{3,3r-3}(H_{n,p}^r)=O(p^3n^{3r-3})$.  By the Chernoff bound we have a.a.s. that $e(H_{n,p}^r)=\Om(pn^r)$, so if we take $t=p^{2/(2r-3)}n(\log n)^{-1}$, then a.a.s. $t^{5-3r}R_{3,3r-3}(H_{n,p}^r)$ is significantly smaller than $t^{2-r}e^(H_{n,p}^r)$. A similar result holds for each term $t^{2-v}R_{\ell,v}(H_{n,p}^r)$ with $\ell=2,3$ and $v\le \ell(r-1)$, so by Lemma~\ref{lem:homTri} we conclude $\ex(H_{n,p}^r,\C_{[3]}^r)]\geq p^{1/(2r - 3)}n^{2 - o(1)}$ a.a.s., proving the lower bound in Theorem \ref{thm:randTri}.  \hfill $\blacksquare$

\medskip

We note that the proof of Lemma~\ref{lem:homTri} fails for larger $\ell$.  In particular, a Berge 4-cycle can appear in $H''$ by mapping onto two edges in $J$ intersecting at a single vertex, and with this the bound becomes ineffective.

\section{Concluding remarks}

$\bullet$ In this paper, we extended ideas of Balogh and Li to bound the number of $n$-vertex $r$-graphs with $m$ edges and girth more
than $\ell$ in terms of the number of $n$-vertex graphs with $m$ edges and girth more than $\ell$. The reduction is best possible when $m = \Theta(n^{\ell/(\ell - 1)})$
and $\ell - 2$ divides $r - 2$. Theorem~\ref{thm:reduceR3} shows that similar reductions can be made when forbidding a single family of Berge cycles.  

By using variations of our method, we can prove the following generalization. For a graph $F$, a hypergraph $H$ is a \textit{Berge}-$F$ if there exists a bijection $\phi:E(F)\to E(H)$ such that $e\sub \phi(e)$ for all $e\in E(F)$. Let $\c{B}^r(F)$ denote the family of $r$-uniform Berge-$F$.  We can prove the following extension of Theorem \ref{thm:reduceR3}: if there exists a vertex $v\in V(F)$ such that $F-v$ is a forest, then there exists $c = c(F,r)$ such that
\begin{align*}
	\N_m^r(n,\c{B}^r(F))&\le 2^{cm}\cdot \N_{[m]}^2(n,F)^{r!/2}.
\end{align*}
For example, this result applies when $F$ is a theta graph.  We do not believe that the exponent $r!/2$ is optimal in general, and we propose the following problem.
\begin{problem}
	Let $\ell,r \geq 3$. Determine the smallest value $\be=\be(\ell,r)>0$ such that there exists a constant $c=c(\ell,r)$ so that, for all $m,n \geq 1$, \[\N_m^r(n,\C_\ell^r) \le 2^{cm}\cdot \N_{[m]}^2(n,C_{\ell})^{\be}.\]
\end{problem}
Theorem~\ref{thm:reduceR3} shows that $\be\le r!/2$ for all $\ell,r$, but in principle we could have $\be=O_\ell(r)$.  We claim without proof that it is possible to use variants of our methods to show $\be(3,r),\be(4,r)\le {r\choose 2}$, but beyond this we do not know any non-trivial upper bounds on $\be$.

\medskip

$\bullet$ We proposed Conjecture \ref{conj:rand} on the extremal function for subgraphs of large girth in random hypergraphs:
for some constant $\gamma = \gamma(\ell,r)$, a.a.s.:
\[ \ex(H_{n,p}^r,\C_{[\ell]}^r) = \left\{\begin{array}{ll}
n^{1 + \frac{1}{\ell - 1} + o(1)} & n^{-r + 1 + \frac{1}{\ell - 1}} \leq p < n^{-\frac{\gamma(\ell - 1 - k)}{\ell - 1}}, \\
p^{\frac{1}{\gamma k}} n^{1 + \frac{1}{k} + o(1)} &  n^{-\frac{\gamma(\ell - 1 - k)}{\ell - 1}} \leq p \leq 1.
\end{array}\right.\]
For $\ell = 3$, this conjecture is true with $\gamma = 2r - 3$, and Conjecture \ref{conj:bestexponent} suggests perhaps $\gamma = r - 1 + (r - 2)/(\ell - 2)$, although
we do not have enough evidence to support this (see also the work of Mubayi and Yepremyan~\cite{MY} on loose even cycles). It would be interesting as a test case to know if $\gamma(3,4) = 5/2$:

\begin{problem}
Prove or disprove that Conjecture \ref{conj:rand} holds with $\gamma(3,4) = 5/2$.
\end{problem}

$\bullet$ It seems likely that $\N_m^r(n,\c{F})$ controls the a.a.s. behavior of $\ex(H_{n,p}^r,\c{F})$ as $n \rightarrow \infty$. Specifically,
it is clear that if $\c{F}$ is a family of finitely many $r$-graphs and $p = p(n)$ and $m = m(n)$ are defined so that
$p^m \N_m^r(n,\c{F}) \rightarrow 0$ as $n \rightarrow \infty$, then a.a.s. as $n \rightarrow \infty$, $H_{n,p}^r$ contains no $\c{F}$-free subgraph with at least $m$ edges.
It would be interesting to determine when $H_{n,p}^r$ a.a.s contains an $\c{F}$-free subgraph with at least $m$ edges. In particular, we leave the following problem:

\begin{problem}
Let $m = m(n)$ and $p = p(n)$ so that $p^m \N_m^r(n,\ell) \rightarrow \infty$ as $n \rightarrow \infty$. Then $H_{n,p}^r$ a.a.s contains a subgraph of girth more than $\ell$ with at least $m$ edges.
\end{problem}

In particular, perhaps one can obtain good bounds on the variance of $\N_m^r(n,\ell)$ in $H_{n,p}^r$.

\textbf{Acknowledgments.}  We thank the referees for their helpful comments, as well as Emanuele Natale and \'Edouard Oyallon for pointing out an error in our statement of the upper bound for Theorem~\ref{thm:randGirth} when $\ell$ is odd.

\bibliographystyle{abbrv}
\bibliography{GT}

\section*{Appendix: Proof of Theorem~\ref{thm:MS}}
Here we give a formal proof of Theorem~\ref{thm:MS}.  The key tool will be the following theorem of Morris and Saxton.
\begin{theorem}[\cite{MS} Theorem 5.1]\label{thm:MSTech}
	For each $k\ge 2$, there exists a constant $C=C(k)$ such that the following holds for sufficiently large $t,n\in \mathbb{N}$ with $t\le n^{(k-1)^2/k(2k-1)}/(\log n)^{k-1}$.  There exists a collection $\c{G}_k(n,t)$ of at most 
	\[\exp(Ct^{-1/(k-1)}n^{1+1/k}\log t)\]
	graphs on $[n]$ such that $e(G)\le t n^{1+1/k}$ for all $G\in \c{G}_k(n,t)$ and such that every $C_{2k}$-free graph is a subgraph of some $G\in \c{G}_k(n,t)$.
\end{theorem}
Recall that we wish to prove that for $\ell\ge 3$ and $k = \lfloor \ell/2\rfloor$, there exists a constant $c>0$ such that if $n$ is sufficiently large and $m\ge n^{1+1/(2k - 1)}(\log n)^2$, then
\[\N_m^2(n,\C_{[\ell]}) \le e^{cm}(\log n)^{(k - 1)m}\l(\f{n^{1+1/k}}{m}\r)^{k m}.\]
The bound is trivial if $\ell=3$ since $\N_m^2(n,C_3)\le {n^2\choose m}$, so we may assume $\ell\ge 4$ from now on.  Because $\N_m^2(n,\C_{[\ell]})\le \N_m^2(n,C_{2k})$ for all $\ell\ge 4$, it suffices to prove this bound for $\N_m^2(n,C_{2k})$.  For any integer $t\le n^{(k-1)^2/k(2k-1)}/(\log n)^{k-1}$ and $n$ sufficiently large, Theorem~\ref{thm:MSTech} implies
\begin{equation}\N_m^2(n,C_{2k})\le |\c{G}_k(n,t)|\cdot {t n^{1+1/k}\choose m}\le\exp(Ct^{-1/(k-1)}n^{1+1/k}\log t)\cdot (etn^{1+1/k}/m)^m,\label{eq:MS}\end{equation}
with the first inequality using that every $C_{2k}$-free graph on $m$ edges is an $m$-edged subgraph of some $G\in \c{G}_k(n,t)$.  By taking $t=(n^{1+1/k}\log n/m)^{k-1}$, which is sufficiently small to apply \eqref{eq:MS} provided $m\ge n^{1+1/(2k - 1)}(\log n)^2$, we see that $\N_m^2(n,C_{2k})$ satisfies the desired inequality. \hfill $\blacksquare$
\end{document}